\documentclass[a4paper,9pt]{article}

\usepackage[english]{babel}
\usepackage[utf8]{inputenc}
\usepackage{amsfonts}
\usepackage{amsmath}
\usepackage{graphicx}
\usepackage[colorinlistoftodos]{todonotes}
\usepackage{tikz-cd}
\usepackage{amsthm}
\usepackage{lipsum}
\usepackage{epigraph}
\theoremstyle{definition}

\newtheorem{remark}{Remark}
\theoremstyle{plain}
\newtheorem{proposition}{Proposition}
\newtheorem{theorem}{Theorem}

\DeclareMathOperator{\lt}{LT}

\DeclareMathOperator{\lc}{LC}

\DeclareMathOperator{\Ima}{Im}
\DeclareMathOperator{\Tor}{Tor}
\DeclareMathOperator{\Ext}{Ext}

\DeclareMathOperator{\tl}{TL}

\title{Constructing a free resolution and few remarks on bar homology of Temperley-Lieb Algebra $\tl_3$ }
\author{SOUTRIK ROY CHOWDHURY}
\date{}

\begin{document}
\maketitle
\begin{abstract}
In this paper I attempt to compute the Anick's resolution of Temperley-Lieb algebra $\tl_3$ and then I compute the bar homology of $\tl_3$ or equivalently the $\Tor_\ast^{\tl_3}(\mathbb{C},\mathbb{C})$. It shows that the differentials in the formulae of Anick's resolution and of the bar homology depends on $\tau$ and thus $\Tor_\ast^{\tl_3}(\mathbb{C},\mathbb{C})$ depends on $\tau$.
\end{abstract}
For $\tau \in \mathbb{C}$, the Temperley-Lieb algebra $\tl_n$ is an associative $\mathbb{C}$ -algebra generated by $ 1, e_1, e_2, \dots, e_{n-1} $ modulo the relations:
\begin{itemize}
\item $e_ie_j = e_je_i$ for $|i-j| > 1$
\item $e_ie_je_i = e_i$ for $|i-j| = 1$
\item $e_ie_i = \tau e_i$.
\end{itemize}
We simply write the algebra $\tl_n$ when $\tau$ is understood. Therefore the structure of the algebra $\tl_3$ is $\{ 1, e_1,e_2 | e_1e_2e_1 - e_1, e_2e_1e_2 - e_2, e_1e_1 - \tau e_1, e_2e_2 - \tau e_2 \}$. The definition of $\tl_n$ can be described by tangle diagram and there exists maps between braid group $B_n$ and $\tl_n$ which can also be motivated by tangle diagram. In appendix I write about it in a little details where I also compute the Gr\"{o}bner basis for $B_3$.\\
One of many ways to compute a free resolution of a graded augmented algebra is resolution constructed by Anick back in 1986 [1]. This resolution shows nice combinatorial construction of the homology classes of the algebra where we actively involve computed Gr\"{o}bner basis of the algebra. In this paper I attempt to construct Anick's resolution for algebra $\tl_3$ and reach some conclusion regarding the general structure of the resolution. Then I compute bar homology of $\tl_3$ from the computed Anick's resolution and also in that case make remark on the general structure of bar homology. It is observed that both are dependent on $\tau$.\\ 
 We recall the formula for Anick's resolution:
 \begin{theorem}[Anick's resolution]
Let $A$ be a graded algebra with augmentation (i.e. there exists an augmentation map $\epsilon : A \rightarrow K$), and let $C_n$ be the set of $n$-chains. We have a resolution of $A$ of the following form:
\[ \begin{tikzcd}
\dots \arrow{r}{d_{n+1}} & C_n \otimes A \arrow{r}{d_n} & C_{n-1} \otimes A \arrow{r}{d_{n-1}} & C_{n-2} \otimes A \arrow{r}{d_{n-2}} & \dots \\
 \dots \arrow{r}{d_2} & C_1 \otimes A \arrow{r}{d_1} & C_0 \otimes A \arrow{r}{d_0} & C_{-1} \otimes A \arrow{r}{\epsilon} & K \arrow{r} & 0
\end{tikzcd}\]
with splitting inverse maps $i_n : \ker d_{n-1} \rightarrow C_n \otimes A$ (which, unlike $d_n$ need not to be homomorphisms of modules). Where:
\begin{itemize}
\item $d_0( x \otimes 1) = 1 \otimes x$.
\item $i_{-1}(1) = 1 \otimes 1$.
\item $i_0( 1 \otimes x_{i_1}x_{i_2}\dots x_{i_n}) = x_{i_1} \otimes x_{i_2}\dots x_{i_n}$.
\item $d_{n+1}(gt \otimes 1) = g \otimes t - i_nd_n( g \otimes t)$ for all ($n+1$)-chains $gt$, with tail $t$.
\item $i_n(u) = \alpha g \otimes c + i_n( u - \alpha d_n(g \otimes c))$ for all $u \in \ker d_{n-1}$ with leading term $f \otimes s$, where $\overline{fs} = \overline{gc}$, and $f$ is a ($n-1$)-chain and $g$ is a $n$-chain and $\alpha =\lc(f \otimes s)$. The bar over $fs$ and $gc$ denotes reduction of them to normal forms.
\end{itemize}
\end{theorem} 
For proof of the theorem 1 we refer to [2]. For more details regarding the resolution and concept of chains, leading terms $\lt$ and leading coefficients $\lc$ we refer to standard text [2,3].\\
Assuming the DEGLEX order $1 < e_1 < e_2$, Gr\"{o}bner basis of $\tl_3$ is given by the relations themselves i.e. $\{e_1e_2e_1 - e_1, e_2e_1e_2 - e_2, e_1e_1 - \tau e_1, e_2e_2 - \tau e_2 \}$. It is easy to verify that the S -polynomial between $e_1e_2e_1$ and $e_2e_1e_2$ is $0$. Similarly the S -polynomials between $e_1e_2e_1$ and $e_1e_1$ and $e_2e_1e_2$ and $e_2e_2$ in pairs respectively are $0$. Therefore the relations themselves satisfy Bergman's Diamond Lemma [4] and hence form the Gr\"{o}bner basis of the algebra $\tl_3$.\\
Elements of chain $C_0$ are $1$, $e_1$, $e_2$.\\
Elements of chain $C_1$ are $e_1e_2e_1, e_2e_1e_2, e_1e_1, e_2e_2$.\\
Elements of chain $C_2$ are $e_1e_2e_1e_1, e_1e_2e_1e_2, e_2e_1e_2e_2, e_2e_1e_2e_1, e_1e_1e_2e_1$,\\ $e_2e_2e_1e_2, e_1e_1e_1, e_2e_2e_2$.\\
Elements of chain $C_3$ are $e_1e_2e_1e_1e_1, e_1e_2e_1e_1e_2e_1, e_1e_2e_1e_2e_2, e_1e_2e_1e_2e_1e_2,$\\ $ e_2e_1e_2e_2e_2, e_2e_1e_2e_2e_1e_2, e_2e_1e_2e_1e_1, e_2e_1e_2e_1e_2e_1$,\\ $e_1e_1e_2e_1e_2, e_1e_1e_2e_1e_1, e_1e_1e_1e_1, e_1e_1e_1e_2e_1,$\\ $e_2e_2e_2e_2, e_2e_2e_2e_2e_1e_2, e_2e_2e_1e_2e_1, e_2e_2e_1e_2e_2$.\\
Rest of chains are constructed in the similar fashion. Though its a bit difficult to write the exact formulae for chains but the construction involves nice combinatorics and we deduce the following proposition:
\begin{proposition}
The number of elements in chain $C_{n+1}$ are $2$ times the number of elements in chain $C_n$ for $n \geq 1$. 
\end{proposition}
\begin{proof}
This is easy to verify. Each element in $C_n$ generates two elements for $C_{n+1}$ following the defined way of construction of chains [3]. 
\end{proof}
Now I would like to make remarks on the length of elements in chain $C_n$ which is essential when we would like to construct the Hilbert series for $\tl_3$ using chains [3]. I don't compute the Hilbert series for $\tl_3$ here but using the remarks one can construct it easily.
\begin{remark}
In $C_n$ when $n$ is even ($n \geq 2$) elements of length (sometime we call them degree instead of length) $(n+1), (n+2), (n+3), \dots, \frac{3n+2}{2}$ are present.
\end{remark}
\begin{remark}
In $C_n$ when $n$ is odd ( $n \geq 3$) elements of length $(n+1), (n+2), (n+3), \dots, \frac{3(n+1)}{2}$ are present.
\end{remark}
The verification of these two remarks depends on the way we take a careful look in the construction of chains and its easy to check them.
\newpage
Now I compute the Anick's resolution (theorem 1) of the algebra $\tl_3$. In theorem 1, field $K$ will become $\mathbb{C}$ and $A$ is $\tl_3$. Namely the following formulae hold:
\[ d_0(e_1 \otimes 1) = 1 \otimes e_1 \hspace{5mm} d_0(e_2 \otimes 1) = 1 \otimes e_2 \]
Now $d_1: C_1 \otimes \tl_3 \rightarrow C_0 \otimes \tl_3$ are given by
\[ d_1( e_1e_2e_1 \otimes 1) = e_1 \otimes e_2e_1 - e_1 \otimes 1 \]
\[ d_1(e_2e_1e_2 \otimes 1) = e_2 \otimes e_1e_2 - e_2 \otimes 1 \]
\[ d_1(e_1e_1 \otimes 1) = e_1 \otimes e_1 - \tau e_1 \otimes 1 \]
\[ d_1(e_2e_2 \otimes 1) = e_2 \otimes e_2 - \tau e_2 \otimes 1 \]
The differential $d_2: C_2 \otimes \tl_3 \rightarrow C_1 \otimes \tl_3$ is given by:
\[ d_2(e_1e_2e_1e_1 \otimes 1) = e_1e_2e_1 \otimes e_1 - \tau e_1e_2e_1 \otimes 1 + e_1e_1 \otimes 1 \]
\[ d_2(e_1e_2e_1e_2 \otimes 1) = e_1e_2e_1 \otimes e_2 \]
\[ d_2(e_2e_1e_2e_2 \otimes 1) = e_2e_1e_2 \otimes e_2 - \tau e_2e_1e_2 \otimes 1 + e_2e_2 \otimes 1 \]
\[ d_2(e_2e_1e_2e_1 \otimes 1) = e_2e_1e_2 \otimes e_1 \]
\[ d_2(e_1e_1e_2e_1 \otimes 1) = e_1e_1 \otimes e_2e_1 + \tau e_1e_2e_1 \otimes 1 - e_1e_1 \otimes 1 \]
\[ d_2(e_2e_2e_1e_2 \otimes 1) = e_2e_2 \otimes e_1e_2 + \tau e_2e_1e_2 \otimes 1 - e_2e_2 \otimes 1 \]
\[ d_2(e_1e_1e_1 \otimes 1) = e_1e_1 \otimes e_1 \]
\[ d_2(e_2e_2e_2 \otimes 1) = e_2e_2 \otimes e_2 \]
For $d_3: C_3 \otimes \tl_3 \rightarrow C_2 \otimes \tl_3$ the formulae look like this:
\[ d_3(e_1e_2e_1e_1e_1 \otimes 1) = e_1e_2e_1e_1 \otimes e_1 - e_1e_1e_1 \otimes 1 \]
\[ d_3(e_1e_2e_1e_1e_2e_1 \otimes 1) = e_1e_2e_1e_1 \otimes e_2e_1 + \tau e_1e_2e_1e_2 \otimes e_1 - e_1e_2e_1e_1 \otimes 1 - e_1e_1e_2e_1 \otimes 1 \]
\[ d_3(e_1e_2e_1e_2e_2 \otimes 1) = e_1e_2e_1e_2 \otimes e_2 - \tau e_1e_2e_1e_2 \otimes 1 \]
\[ d_3(e_1e_2e_1e_2e_1e_2 \otimes 1) = e_1e_2e_1e_2 \otimes e_1e_2 - e_1e_2e_1e_2 \otimes 1 \]
\[ d_3(e_2e_1e_2e_2e_2 \otimes 1) = e_2e_1e_2e_2 \otimes e_2 - e_2e_2e_2 \otimes 1 \]
\[ d_3(e_2e_1e_2e_2e_1e_2 \otimes 1) = e_2e_1e_2e_2 \otimes e_1e_2 + \tau e_2e_1e_2e_1 \otimes e_2 - e_2e_1e_2e_2 \otimes 1 - e_2e_2e_1e_2 \otimes 1 \]
\[ d_3( e_2e_1e_2e_1e_1 \otimes 1) = e_2e_1e_2e_1 \otimes e_1 - \tau e_2e_1e_2e_1 \otimes 1 \]
\[ d_3(e_2e_1e_2e_1e_2e_1 \otimes 1) = e_2e_1e_2e_1 \otimes e_2e_1 - e_2e_1e_2e_1 \otimes 1 \]
\[ d_3(e_1e_1e_1e_1 \otimes 1) = e_1e_1e_1 \otimes e_1 - e_1e_1e_1 \otimes 1 \]
\[ d_3(e_1e_1e_1e_2e_1 \otimes 1) = e_1e_1e_1 \otimes e_2e_1 - e_1e_1e_1 \otimes 1 \]
\[ d_3(e_2e_2e_2e_2 \otimes 1) = e_2e_2e_2 \otimes e_2 - e_2e_2e_2 \otimes 1 \]
\[ d_3(e_2e_2e_2e_1e_2 \otimes 1) = e_2e_2e_2 \otimes e_1e_2 - e_2e_2e_2 \otimes 1 \]
\[ d_3(e_1e_1e_2e_1e_2 \otimes 1) = e_1e_1e_2e_1 \otimes e_2 - \tau e_1e_2e_1e_2 \otimes 1 \]
\[ d_3(e_1e_1e_2e_1e_1 \otimes 1) = e_1e_1e_2e_1 \otimes e_1 - \tau e_1e_2e_1e_1 \otimes 1 - \tau e_1e_1e_2e_1 \otimes 1 + e_1e_1e_1 \otimes 1 \]
\[ d_3(e_2e_2e_1e_2e_1 \otimes 1) = e_2e_2e_1e_2 \otimes e_1 - \tau e_2e_1e_2e_1 \otimes 1 \]
\[ d_3(e_2e_2e_1e_2e_2 \otimes 1) = e_2e_2e_1e_2 \otimes e_2 - \tau e_2e_1e_2e_2 \otimes 1 - \tau e_2e_2e_1e_2 \otimes 1 + e_2e_2e_2 \otimes 1 \]
\begin{remark}
So we see how to construct formula $d_{n+1}$ using information available from $d_n$ and $i_n$. It is also seen that most formula are identical i.e. one can get one formula from other just by replacing $e_1$ by $e_2$ and vice versa. It is possible as the construction of there chains are identical. I didn't write the general formula $d_n: C_n \otimes \tl_3 \rightarrow C_{n-1} \otimes \tl_3$ but it can be constructed using combinatorics. The construction will be identical to one computed in [ex 4.3.1, 2]. We see that each differential shows nice combinatorial interpretation of the structure of the chains and hence of the algebra.
\end{remark}
\begin{proposition}
In each $d_n: C_n \otimes \tl_3 \rightarrow C_{n-1} \otimes \tl_3$ for $n \geq 1$ some formulae depend on $\tau$. Therefore differentials $d_n$ in Anick's resolution for all $n$ depend on $\tau$.
\end{proposition}
\begin{proof}
Upto $d_3$ we see that formulae are dependent on $\tau$. It appears as there exist relations $e_1e_1 - \tau e_1$ and $e_2e_2 - \tau e_2$ and when we start computing the formula for $d_n$ for $n \leq 3$ we find $i_n$ where some time leading terms contain $\tau$ as leading coefficients. That's why we find such $\tau$ in $d_1$, $d_2$ and in $d_3$. For higher $d_n$ for $n >3$ as formula depends on previous one and on $i_n$ which is an inverse formula $d_n$ must depends on $\tau$ as in most cases as leading coefficients $\tau$ appear.
\end{proof}
To compute the bar homology of $\tl_3$ or equivalently $\Tor_\ast^{\tl_3}(\mathbb{C},\mathbb{C})$ we need to compute the homology of the complex $\{ (\mathbb{C}C_n \otimes_\mathbb{C} \tl_3) \otimes_{\tl_3} \tl_3, \bar{d_n} \}$ where the induced differential $\bar{d_n} = d_n \otimes 1$ only keeps those elements which are not annihilated by the augmentation of $\tl_3$. In other words under the identification $ (\mathbb{C}C_n \otimes_\mathbb{C} \tl_3) \otimes_{\tl_3} \tl_3 \cong \mathbb{C}C_n$ we have
\[ \bar{d_0}(e_1) = 0, \hspace{3mm} \bar{d_0}(e_2) = 0 \]
\[ \bar{d_1}(e_1e_2e_1) = -e_1, \hspace{2mm} \bar{d_1}(e_2e_1e_2) = -e_2, \hspace{2mm} \bar{d_1}(e_1e_1) = -\tau e_1, \hspace{2mm} \bar{d_1}(e_2e_2) = -\tau e_2 \]
\[ \bar{d_2}(e_1e_2e_1e_2) = \bar{d_2}(e_2e_1e_2e_1) = \bar{d_2}(e_1e_1e_1) = \bar{d_2}(e_2e_2e_2) = 0 \]
\[ \bar{d_2}(e_1e_2e_1e_1)= -\tau e_1e_2e_1 + e_1e_1, \hspace{2mm} \bar{d_2}(e_2e_1e_2e_2) = -\tau e_2e_1e_2 + e_2e_2 \]
\[ \bar{d_2}(e_1e_1e_2e_1) = \tau e_1e_2e_1 - e_1e_1, \hspace{2mm} \bar{d_2}(e_2e_2e_1e_2) = \tau e_2e_1e_2 - e_2e_2 \]
Similarly $\bar{d_3}: C_3 \rightarrow C_2$ have non-zero terms
\[ \bar{d_3}(e_1e_2e_1^3) = -e_1^3, \hspace{2mm} \bar{d_3}((e_1e_2e_1)^2) = -e_1e_2e_1^2 - e_1^2e_2e_1, \hspace{2mm} \bar{d_3}(e_1e_2e_1e_2^2 = -\tau (e_1e_2)^2 \]
\[ \bar{d_3}((e_1e_2)^3) = -(e_1e_2)^2, \hspace{2mm} \bar{d_3}(e_2e_1e_2^3) = -e_2^3, \hspace{2mm} \bar{d_3}((e_2e_1e_2)^2) = -e_2e_1e_2^2 - e_2^2e_1e_2 \]
\[ \bar{d_3}(e_2e_1e_2e_1^2) = -\tau (e_2e_1)^2, \hspace{2mm} \bar{d_3}((e_2e_1)^3)= -(e_2e_1)^2, \hspace{2mm} \bar{d_3}(e_1)^4 = -e_1^3, \hspace{2mm} \bar{d_3}(e_1^3e_2e_1) = -e_1^3 \]
\[ \bar{d_3}(e_2^4) = -e_2^3, \hspace{2mm} \bar{d_3}(e_2^3e_1e_2) = -e_2^3, \hspace{2mm} \bar{d_3}(e_1^2e_2e_1e_2) = -\tau (e_1e_2)^2, \] \[ \bar{d_3}(e_1^2e_2e_1^2) = -\tau e_1e_2e_1^2 - \tau e_1^2e_2e_1 + e_1^3 \]
\[ \bar{d_3}(e_2^2e_1e_2e_1) = -\tau (e_2e_1)^2, \hspace{2mm} \bar{d_3}(e_2^2e_1e_2^2) = -\tau e_2e_1e_2^2 - \tau e_2^2e_1e_2 + e_2^3 \]
\begin{remark}
In similar ways we can compute the remaining $\bar{d_n}$. We see that $\Ima(\bar{d_n})$ depends on $\tau$ and hence the homology of the complex depends on $\tau$ and so does $\Tor_\ast^{\tl_3}(\mathbb{C},\mathbb{C})$.
\end{remark}
Computation of Anick's resolution for other $\tl_n$ will be a little difficult as that time we need to encounter the first relation too but soon we can find formulae using combinatorial relationships. It would be interesting to compute bar homology for $\tl_n$ for $n \geq 4$ and it can be seen that $\Tor_\ast^{\tl_n}(\mathbb{C}, \mathbb{C})$ depends on $\tau$.\\
Another interesting exercise will be to compute $A_\infty$ -algebra structure associated to $\Ext$ -algebra of Temperley-Lieb algebra $\tl_n$. At this stage our computed Anick's resolution is not minimal but using perturbation methods to find minimal resolution from non-minimal Anick's resolution given in [6] and using methods of Merkuluv's construction [7] we can construct such higher homotopy algebra structure for $\tl_n$.
\section*{ Appendix:Relationship of $B_n$ with $\tl_n$}
Gr\"{o}bner basis of 3 strand braid group $B_3 = \langle \sigma_1, \sigma_2 | \sigma_1\sigma_2\sigma_1 - \sigma_2\sigma_1\sigma_2 \rangle$ is given by the set
\[ \{ \sigma_1\sigma_2\sigma_1 - \sigma_2\sigma_1\sigma_2 \}\hspace{1mm} \bigcup_{n=2}^\infty \{ -\sigma_2\sigma_1\sigma_2^2\sigma_1^{n-1} + (-1)^n \sigma_1\sigma_2^n\sigma_1\sigma_2 \} \]
which indeed satisfies the Bergman's Diamond lemma. There is a map from $B_n$ to $\tl_n$ given by
\[ \sigma_i \mapsto A + A^{-1}e_i \]
\[ \sigma_i^{-1} \mapsto A^{-1} + Ae_i \]
where $A \in \mathbb{C}$ such that $\tau = -A^2 - A^{-2}$. The definition of $\tl_n$ can be motivated in terms of tangle diagrams in $\mathbb{R} \times I$ . These
are similar to knot diagrams, except that they can include arcs with endpoints
on $\mathbb{R} \times \{0, 1\}$. Two tangles are considered the same if they are related by a sequence of isotopies and Reidemeister moves of the second and third type. The third
relation in the Temperley-Lieb algebra allows one to delete a closed loop at the
expense of multiplying by $\tau$ . Using these definitions, the map from $B_n$ to $\tl_n$
is given by resolving all crossings using the Kauffman skein relation.\\
It will de interesting to find whether representation of $B_n$ over $\tl_n$ is faithful or not for $n \geq 4$ when $\tau$ is transcendental. Though its a different context but maybe its worthwhile to connect Gr\"{o}bner basis of $B_n$ with bar homology of $\tl_n$ and existing research [5] to find the answer of the question. This will connect representation theory and homological algebra with low dimensional topology.

\vspace{4mm}
ITI Road, Jyotinagar, 2nd Mile, Siliguri-734001,\\
email: roychowdhurysoutrik@gmail.com

\end{document}